\documentclass[oneside]{amsart}
\usepackage{amssymb}

\numberwithin{equation}{section}

\newtheorem{theorem}{Theorem}[section]

\newtheorem{lemma}[theorem]{Lemma}

\theoremstyle{definition}

\theoremstyle{definition}
\newtheorem{definition}[theorem]{Definition}
\newtheorem{df}[theorem]{Definition}

\newtheorem*{definition*}{Definition}

\newtheorem{problem}{Problem}
\newtheorem{observation}{Observation}

\theoremstyle{remark}

\newtheorem*{remark*}{Remark}

\newcommand{\R}{\mathbb R}
\newcommand{\ep}{\varepsilon}
\DeclareMathOperator{\len}{length}
\DeclareMathOperator{\dist}{dist}

\begin{document}

\title
[Polyhedral Finsler spaces with locally unique geodesics]
{Polyhedral Finsler spaces \\ with locally unique geodesics}

\author{Dmitri Burago}                                                          
\address{Dmitri Burago: Pennsylvania State University,                          
Department of Mathematics, University Park, PA 16802, USA}                      
\email{burago@math.psu.edu}                                                     
                                                                                
\author{Sergei Ivanov}
\address{Sergei Ivanov:
St.~Petersburg Department of Steklov Mathematical Institute,
Russian Academy of Sciences,
Fontanka 27, St.Petersburg 191023, Russia}
\email{svivanov@pdmi.ras.ru}

\thanks{The first author was partially supported                                
by NSF grant  DMS-1205597.
The second author was partially supported by
RFBR grant 11-01-00302-a.}



\begin{abstract}
We study Finsler PL spaces, that is simplicial complexes glued out of simplices 
cut off from some normed spaces.
We are interested in the class of Finsler PL spaces featuring local uniqueness
of geodesics (for complexes made of Euclidean simplices, this property is 
equivalent to local CAT(0)).
Though non-Euclidean normed
spaces never satisfy CAT(0), it turns out that they share many common
features. In particular, a globalization theorem holds:
in a simply-connected Finsler PL space local uniqueness of geodesics
implies the global one.
However the situation is more delicate here:
some basic convexity properties do not extend to the PL Finsler case.
\end{abstract}

\maketitle

\section{Preliminaries and discussion}

In this paper we discuss similarities and differences between the
geometry of polyhedral (PL) spaces built out of Euclidean simplices
and simplices in normed spaces.
Generally speaking, we are interested in an analog
of nonpositive curvature for Finsler PL spaces and for
 Finsler geometry in general (see \cite{BH}, \cite{BBI} and \cite{BCS} for
the definition of nonpositive curvature and basics in metric and Finsler geometry).

\begin{definition}\label{d:nonfocusing}
We say that a PL space 
is \textit{non-focusing} if it has locally unique minimal geodesics,
i.e., if every point of $X$ has a neighborhood such that
any two points in this neighborhood are connected by
a unique shortest path.
\end{definition}

For Euclidean PL spaces, this condition is equivalent
to nonpositive curvature.
The Globalization Theorem for CAT(0) spaces implies that for simply connected
Euclidean PL spaces the local uniqueness of geodesics implies the global one:
the space is CAT(0) and therefore every two points are connected by a
unique geodesic.
Note however that the only CAT(0) normed spaces (as well as the only normed
spaces with any curvature bound) are Euclidean spaces. 
In terms of the results, the main assertion proven in this note 
is an  analog of the Globalization Theorem for 
Finsler PL spaces (Theorem \ref{glob}).
However the discussion of related aspects of Finsler 
PL geometry maybe equally important.

Let us first fix some basic terminology. 
A \textit{Finsler PL space} is a
length metric space made of a collection of convex
polyhedral sets in normed spaces glued along faces in such a way 
that faces are identified along affine isometries.
(By \textit{faces} we mean faces of all dimensions,
not just the maximal one.)
We assume that all norms are  $C^1$-smooth and strictly convex
(in the sense that the boundary of the unit ball does not
contain straight line segments).
We say that a PL space is Euclidean if all norms are Euclidean. 
We consider only locally finite PL
spaces, that is the spaces where every point belongs to finitely many faces. Now we can already
formulate the Globalization Theorem for Finsler PL spaces:

A curve $\gamma$ in a length space is a \textit{geodesic} if it is locally minimazing,
that is, every point $t_0$
in the domain of $\gamma$ has a neighborhood such that
$\gamma$ is a shortest path when restricted to that neighborhood.

\begin{theorem}
\label{glob}
In a simply connected non-focusing Finsler PL space,
there is only one geodesic between every two points.
\end{theorem} 

Let us note first that even gluing two normed half-spaces along
an isometry between the boundary hyperplanes is not nearly as an innocent operation 
as in  Euclidean geometry. If we glue two halfÐspaces, the resulting
space remains non-focused (this easily follows from the fact that
the distance function to a point restricted to the hyperplane along which the
half-spaces are glued is convex). However, if we look at slightly more
delicate properties, the situation is not at all that straightforward. 
Here are a few observations.

One of the key features of spaces with curvature bounds are various convexity
properties of distance functions.
As it is mentioned above, a non-Euclidean normed space is never CAT(0), however distance
functions in such spaces enjoy some convexity properties. Namely, the distance to a point 
restricted to any line is a convex function. Furthermore, the distance functions satisfy the
Busemann convexity condition \cite{Busemann}. Namely, for every triangle, the distance between the midpoints of two
sides is no larger than half of the base.  Here is the first

\begin{observation} The Busemann convexity  and even the convexity of the distance function 
to a geodesic  may fail in a space resulting from gluing two normed
half-spaces along an isometry between their boundaries.
\end{observation}

\begin{proof}[Sketch of proof]
An example can be obtained by defining two different norms on the upper and
lower half-planes. Of course, the norms must agree on horizontal vectors. Make the
norms have different tangents at a horizontal vector but coincide with the
standard Euclidean norm near the vertical direction. Then vertical lines obviously 
remain geodesics, and by the First Variation Formula it is easy to see that the distance 
function to the $y$-axis from any point (different from the origin) in the $x$-axis is not
convex at the origin.
\end{proof}

Another thing that follows from the Busemann Convexity (enjoyed by all CAT(0)
spaces) is that two geodesics from one point diverge at least linearly. Furthermore, the
Busemann convexity prohibits the existence of two distinct geodesics asymptotic to each 
other at both ends
(that is, with the distance going to zero).

\begin{observation}
\label{asymptotic}
There exists a simply connected non-focusing Finsler PL space which contains two
geodesics starting from one point and such that the distance between them stays bounded. 
Furthermore, there are examples with two geodesics asymptotic to each other at both ends, 
as well as of a geodesic with a Jacobi field going to zero in both directions along one geodesic.
\end{observation}

\begin{proof}[Sketch of proof]
Such examples are easily constructed if one thinks in terms of geometric optics.
Of course, media with different refractive indices are not glued along isometries,
but the equations governing the rays in question do not change if the norms (which
in our case are not supposed to be Euclidean) are modified near the direction of
the gluing lines. This means that we can make our gluings isometric by changing the
norms in directions which have no affect on the behavior of geodesics we care of. 

Let us discuss a construction of two geodesics asymptotic to each other at both ends
in a bit more detail. The example is two-dimensional.

We start with a rectangle $[-1,1] \times [-2,2]$ divided by the line $y=x$. We keep the
(standard Euclidean) 
metric in the bottom part and multiply it by a small factor (say, 1.01) in the top one. We trim the 
norm in the top near the direction $x=y$ to make the the gluing isometric; this trimming
would not change the behaviour of geodesics which are nearly vertical.

Next, we glue the bottom side of the rectangle to the top one so that the point $(0,-2)$ is
glued to $(0, 2)$ but the bottom side is slightly stretched (by the same factor 1.01)
in order to make this gluing isometric.

The universal cover of the resulting belt is a (``vertical'') strip, and the line $x=0$ is a geodesic. 
One easily sees that there are nearby geodesics that are exponentially asymptotic
to $x=0$.

Now one simply takes two belts like the one described above and connects 
them by a rectangle attached along the $x$-axis in each of them. The resulting space 
has of course a boundary, but this is easy to attach an outside part which is irrelevant 
to all geodesics in question and does not destroy the non-focusing property.   
Now there are geodesics approaching the vertical lines $x=0$ in each of the strips,
and they are connected via the rectangle gluing the strips, which gives the 
desired example.
\end{proof}

To indicate that Globalization for Finsler PL spaces is not a ``general nonsense"
statement, let us make

\begin{observation}
\label{ob:nonsmooth}
The Globalization Theorem does not hold for Finsler PL spaces with
non-smooth (though still strictly convex) norms.
Geodesic segments may fail to be stable for such spaces.
\end{observation}

\begin{proof}[Sketch of proof]
An example is readily constructed by gluing three two dimensional strips (like a Russian
flag). We keep a Euclidean norm in the top and bottom strips and introduce a (strictly convex)
norm with corners at the vertical direction in the middle strips (it is convenient to think of 
all three norms being symmetric in their vertical axis).  One sees that vertical geodesics 
in the middle strip split as they enter the top and bottom strip (making a fan whose 
angle depends on how sharp the corners are).
Therefore two points $p,q$ in the top and bottom strips lying on the same vertical line
can be connected by a family of geodesics. Namely if $x,y$ are two points on the boundary
of the middle strip such that the segment $[xy]$ is vertical and sufficiently close
to the segment $[pq]$, then the broken line $[pxyq]$ is a geodesic of our metric.

Note that the length is non-constant along this family of geodesics.
In particular, most geodesics in this family are not critical points
of the length functional.
This surprising property is possible due to the lack of differentiability
of the length in the middle strip.
\end{proof}

Let us now speculate on some problems and directions arising from the results of the paper.
One observation is that both polyhedral and smooth 2-dimensional saddle surfaces in affine
spaces form affine classes of sets such that they are non-focusing for any choice of a
(smooth strictly convex) norm in the ambient space. Motivated by these examples, let us say that a subset in an
affine space is {\it universally non-focusing} if it is non-focusing for every smooth
and strictly convex norm in the ambient space. It would be nice to have a description 
of such affine classes of subsets, say in PL or Lipschitz category (cf. \cite{Shefel}). 
For instance, a concrete  question reads:

\begin{problem}
Is this enough to
verify the condition of being universally non-focusing only for all Euclidean norms
in the ambient space?
\end{problem}

Another challenging goal would be to define an analog of CAT(0) in Finsler geometry. There
are many problems which could be approached from this viewpoint, including finding a shortest
path amidst convex obstacles (cf. \cite{BGS}) and the shortest braid problem 
(see \cite{braid-question} and \cite{BI}).
Perhaps one could try to
look at Gromov-Hausdorff limits of non-focusing simply connected Finsler PL spaces.
Note that the limit of norms on just one simplex  maybe not strictly convex, so we need to require 
that all norms are uniformly convex; furthermore, the limit norm on each simplex
may fail to be smooth, giving rise to examples like in Observation~\ref{ob:nonsmooth}.
To avoid such examples, let us assume that the unit spheres of norms are smooth
and their curvatures are uniformly bounded
away from 0 and $\infty$. To make this 
condition affine-invariant we can measure the curvatures with respect to, for instance, the 
Euclidean metric defined by the John ellipsoid of the norm.

\begin{problem}
\label{limits}
Under the above assumptions, is this true that bi-Lipschitz limits
of  non-focusing Finsler PL spaces 
(or, more generally, Gromov--Hausdorff limits with 
systoles bounded away from zero)
also have unique geodesics?
\end{problem}

Finally, let us note that some basic properties of Euclidean PL spaces are not clear at all for 
Finsler PL.
For instance, it is known that every finite Euclidean PL space has a sub-triangulation
that admits a PL isometric map (linear on every simplex)
into a Euclidean space (cf.\ \cite{Zalgaller}, \cite{Krat}). In particular, all simplices in such
a sub-triangulation are convex. It is not clear if Finsler PL spaces admit such sub-triangulations
(even though Finsler manifolds admit isometric embeddings, see for instance \cite{BI93} and
also \cite{Shen}). Hence

\begin{problem}
Does every Finsler PL space (where all norms on simplices are smooth
and quadratically convex) admit a sub-triangulation into convex simplices?
\end{problem}

If this is true, an argument has to make use of quadratic convexity, for 
one can show that there is a 3-dimensional
counter-example  with strictly but not quadratically convex norms (where the problem is actually
localized in a neighborhood of one edge). We leave the details to the reader.

The following three problems have been suggested by an anonymous referee. We find them 
quite interesting. We have slightly reworded them here
from the version the referee had formulated them. We also add some discussion.

\begin{problem}
Given a Finlser PL space,
how to understand whether it is non-focusing?
\end{problem}

This is an attempt to generalize the Euclidean polyhedral case, 
where one has manageable conditions in terms of links and angle defects,
see e.g.~\cite{BH}. 
(Recall that non-focusing is a local property, and
a Euclidean polyhedral space is non-focusing if and only if
it is locally CAT(0).)

\begin{problem}
\label{approximation}
Which smooth Finsler manifolds can be approximated by non-focusing Finsler PL spaces
(compare with Problem~\ref{limits})?
\end{problem}

If the answer to  Problem~\ref{limits} is affirmative then approximable manifolds
must have no conjugate points.
Even if the answer to Problem~\ref{limits} is negative, there are topological pre-conditions
to the existence of approximations. 
For instance, Theorem \ref{glob} implies that the universal cover of
a non-focusing Finsler PL space is contractible.
Hence a non-aspherical space cannot carry
a non-focusing Finsler PL metric.

On the other hand, in any two-dimensional smooth Finsler manifold $M$ every point 
$p\in M$ has a neighborhood which is
bi-Lipschitz approximable by non-focusing polyhedral Finsler metrics. Indeed, by 
\cite[Theorem 1.4]{BI}, a neighborhood of $p$ admits a saddle smooth isometric embedding into a 4-dimensional
normed space $V$. Moreover this embedding is strictly saddle in the sense that
its second fundamental forms are all non-degenerate and indefinite.
A suitable fine triangulation yields a saddle PL surface in $V$
approximating the image of this embedding.
By Theorem~\ref{cone} (see Section~\ref{saddle}),
the metric of a saddle PL surface is non-focusing, hence the result.

\begin{problem}
\label{smoothing}
Given a non-focusing polyhedral Finsler space which is a smooth manifold, is it possible to approximate it
by Finsler manifolds without conjugate points or even of negative flag curvature?
\end{problem}

The example in Observation \ref{asymptotic} suggests that the answer to the flag curvature version
is likely to be negative, at least in dimensions greater than~2.

The Riemannian (vs.\ Euclidean polyhedral) versions of Problems
\ref{approximation} and~\ref{smoothing}
seem to be of interest too.
Note that solutions to the Finsler versions
do not imply the Riemannian counterparts, nor the other way round.

\medskip

The rest of the paper is organized as follows. In Section \ref{stability}, after necessary 
preliminaries, we show that in every Finsler PL space geodesic segments are stable, that 
is every geodesic between two 
points strictly minimizes the length among all $C^0$-close curves (Theorem 
\ref{t-geodesics-are-stable}). In Section \ref{uniqueness} we show that in a proper 
metric space where geodesic  segments are stable and locally unique, they are globally 
unique (Theorem \ref{t-uniqueness-in-length-space}). This immediately implies Theorem
\ref{glob}. Finally, in Section \ref{saddle} we show that simply connected PL
saddle surfaces in normed spaces are non-focusing.

\section{Stability of geodesics}
\label{stability}

We begin with basic definitions and elementary facts. 

Let $X$ be a Finsler PL space
(with a fixed decomposition into faces).
The \textit{star} of a point $p\in X$ is
the union of relative interiors of all faces containing $p$.
Note that the star is open and admits a natural arc-wise isometry
to a neighborhood of the apex of the tangent cone.
A \textit{segment} in $X$ is an affine segment contained in one face of~$X$.
Any two points $x,y\in X$ that belong to one face of $X$ are connected
by a unique segment, which is denoted by $[x,y]$.
A \textit{broken line} is a 
curve composed of a finite sequence of segments so that the end of the
preceding segment coincides with the beginning of the next one.

A Finsler PL space is a \textit{cone} with apex $p$ 
if its faces are convex polyhedral cones with apexes at~$p$.
A cone with apex $p$ has a natural one-parameter
family of dilatations 
which fix $p$ and are homotheties with coefficient $t$ in each face.
It is easy to see that every point $p$ in a Finsler PL space $X$ has a
neighborhood which is isometric to a neighborhood of an apex of a cone
(and $p$ corresponds to the apex). 
This cone is called the \textit{tangent cone} at $p$.
A \textit{radial segment}
in a cone with apex $p$ is a segment such that one of its endpoints is $p$.

\begin{lemma}
\label{l-radial-rays}
In a Finsler PL cone, radial segments are unique minimizers.
\end{lemma}

\begin{proof}
First let us show that radial segments are minimizers. Indeed, in each face, the
distance to the apex $p$ of the cone is a 1-Lipschitz function, and these functions
agree on intersections of faces. This implies that the union of these functions
is 1-Lipschitz in the cone, hence the radial segments minimize length.
To show that they are unique minimizers, assume the contrary.
Then there is a minimizer $\gamma$ between $p$ and some point $a$
such that there is a point $b$ on $\gamma$ arbitrarily close to $a$
and not lying on the radial segment $[pa]$.
Since we have already shown that radial segments in all cones are minimizers,
the distance from $p$ to $b$ is realized by a radial segment
and (assuming that $b$ lies in a conical neighborhood of $a$)
the distance between $a$ and $b$ is realised by a segment $[ab]$,
and we obtain a contradiction with the strict triangle inequality
for $p$, $a$ and $b$ in one of the faces (the triangle inequality
is strict since the norms in the faces are strictly convex).
\end{proof}

\begin{lemma}
\label{l-geodesics-are-broken-lines}
Every geodesic in a Finsler PL space is a broken line.
\end{lemma}

\begin{proof}
This is a standard compactness argument. Every point on the geodesic
has a neighborhood where it minimizes and 
which belongs to a cone, and by the previous lemma it is a segment
or a union of two segments in that
neighborhood. Choosing a finite cover concludes the proof.
\end{proof}

\begin{df}
Let $X$ be a length space and $\gamma\colon [a,b]\to X$
a geodesic segment.
We say that $\gamma$ is \textit{stable}
if it strictly minimizes the length among nearby curves.

More precisely, $\gamma$ is stable
if there is a neighborhood $U$ of $\gamma$
(in the space of all curves in $X$ parameterized by $[a,b]$ and
equipped with the uniform topology)
such that every curve $\gamma_1\in U$ connecting
the same endpoints satisfies $\len(\gamma_1)>\len(\gamma)$
unless $\gamma_1$ is a reparameterization of $\gamma$.
\end{df}

\begin{theorem}
\label{t-geodesics-are-stable}
Every geodesic in a Finsler PL space is stable.
\end{theorem}

\begin{proof}
We begin with some preliminaries from elementary
geometry of a normed space.
Let $(V^n,\|\cdot\|)$ be a normed vector space
whose norm is $C^1$ and strictly convex.
For a unit vector $v\in V$, consider the hyperplane
$v^\perp\subset V$ defined by
$$
 v^\perp = \big\{ w\in V : \tfrac d{dt}\big|_{t=0} \|v+tw\| = 0 \big\} .
$$
In other words, $v^\perp$ is the kernel of the derivative
of the norm at~$v$. We refer to this hyperplane as the
\textit{orthogonal complement} of~$v$.
(Note that the relation $w\in v^\perp$ is not symmetric,
and there is no reasonable symmetric notion of orthogonality
in normed spaces.)

Let $\gamma$ be a straight line in $V$
parameterized by arc length:
$\gamma(t)=a+vt$ where $a,v\in V$, $\|v\|=1$.
For every $t$, consider the affine hyperplane
$H_{\gamma}(t)$ parallel to $v^\perp$ through $\gamma(t)$.
We refer to this hyperplane as the
\textit{orthogonal slice} for $\gamma$ at~$t$,
and say that the family of hyperplanes is the
orthogonal slicing.

Let $q\in H_\gamma(t_0)$. Then
\begin{equation}
\label{e-strict-busemann}
 \|q-\gamma(t)\| \ge |t-t_0|
\end{equation}
for all $t\in\R$,
with equality only if $q=\gamma(t_0)$
 (since $\|\cdot\|$ is strictly convex), and
\begin{equation}
\label{e-busemann}
  \lim_{|t|\to\infty} \big( \|q-\gamma(t)\| - |t-t_0| \big) = 0
\end{equation}
due to the fact that $\|\cdot\|$ is differentiable at~$v$.

If $F\subset V$ is a convex polyhedral set containing
a point $\gamma(t)$,
we denote by $H^F_\gamma(t)$ the intersection
$H_\gamma(t)\cap F$ and refer to this intersection
as an orthogonal slice for $\gamma$ in~$F$.

Now let $\gamma$ be a segment in a Finsler PL space $X$
and $p=\gamma(t)$ an interior point of this segment.
In each $k$-dimensional face $F$ containing $p$
there is a $(k-1)$-dimensional orthogonal slice
$H^F_\gamma(t)$ for $\gamma$ at~$t$.
These slices agree on the intersections of faces
(since the norms on faces agree).
We denote by $H_\gamma(t)$ the union of the sets $H^F_\gamma(t)$
over all faces containing $p$
and refer to this set as the \textit{local orthogonal slice}
for~$\gamma$ at~$t$. Clearly it divides the star of $p$
into two components.

Let $\gamma\colon[a,b]\to X$ be a geodesic.
By Lemma \ref{l-geodesics-are-broken-lines},
$\gamma$ is a broken line.
Let $p_0=\gamma(t_0),p_1,\dots, p_n=\gamma(t_n)$ be its vertices
where $a=t_0\le t_1\le\dots\le t_n=b$.
For each $i=1,\dots,n$ fix a point 
$s_i\in(t_{i-1},t_i)$
and let $H_i$ denote the local orthogonal slice
for the $i$-th edge of $\gamma$ at~$s_i$.

Suppose that $\gamma$ is not stable and let $\gamma'$
be a nearby curve connecting the same endpoints and
such that $\len(\gamma')\le\len(\gamma)$.
We may assume that $\gamma$ and $\gamma'$
have no common initial interval (otherwise take the
point where they first split as the new starting point).
If $\gamma'$ is sufficiently close to $\gamma$,
it intersects the slices $H_i$ at some points
$\gamma'(s'_i)$  where $s'_i\in(t_{i-1},t_i)$.
Furthermore the intervals $\gamma'([a,s'_1])$,
$\gamma'([s'_i,s'_{i+1}])$, $\gamma'([s'_n,b])$ of $\gamma'$ are 
close to the respective intervals of $\gamma$
and hence contained in
the stars of the respective points $p_0$, $p_1$,\dots, $p_n$.

\begin{lemma}
\label{l-between-slices}
$\len(\gamma'|_{[s'_i,s'_{i+1}]}) \ge \len(\gamma|_{[s_i,s_{i+1}]})$
for all $i=1,\dots,n-1$.
\end{lemma}

\begin{proof}
Fix $i\in\{1,\dots,n\}$ and let
$p=\gamma(t_i)$, $p_-=\gamma(s_i)$, $p_+=\gamma(s_{i+1})$,
$p'_-=\gamma'(s'_i)$, $p'_+=\gamma'(s'_{i+1})$.
Let $K$ denote the tangent cone to $X$ at~$p$.
Recall that $\gamma'([s'_i,s'_{i+1}])$ is contained in the star of~$p$.
The star of $p$ admits a natural arc-wise isometry onto a neighborhood
of the apex in~$K$. We abuse notation and use the same letters for
points in the star and their images in the cone.
Extend $\gamma|_{[s_i,s_{i+1}]}$ parameterizing
the union of radial segments $[p,p_+]$ and $[p,p_-]$ in the cone
to a curve $\bar\gamma\colon\R\to K$ made
of two radial rays $\bar\gamma|_{(-\infty,t_i]}$
and $\bar\gamma|_{[t_i,+\infty)}$ through $p_-$ and $p_+$, respectively.
Suppose that
$$
 \len(\gamma'|_{[s'_i,s'_{i+1}]}) = \len(\gamma|_{[s_i,s_{i+1}]}) - \ep
$$
where $\ep>0$. Since $p'_-$ belongs to the star of $p_-$,
there is a face of $K$ containing $p_-$ and $p'_-$.
Applying \eqref{e-busemann} to the norm of this face yields that
$$
 \lim_{t\to-\infty}
 \big( \len[\bar\gamma(t),p'_-] - \len[\bar\gamma(t),p_-] \big) = 0 .
$$
Therefore there exists $t_-\in (-\infty,t_i]$
such that
$$
 \len[\bar\gamma(t_-),p'_-] < \len[\bar\gamma(t_-),p_-] + \frac\ep2 .
$$
Similarly there exists $t_+\in [t_i,+\infty)$
such that
$$
 \len[p'_+,\bar\gamma(t_+)] < \len[p_+,\bar\gamma(t_-)] + \frac\ep2
$$
(here the segments lie in a face containing both $p_+$ and $p'_+$).
These inequalities imply that the curve in $K$ composed from
$\gamma'|_{[s'_i,s'_{i+1}]}$ and the segments $[\bar\gamma(t_-),p'_-]$
and $[p'_+,\bar\gamma(t_+)]$ is shorter than $\bar\gamma|_{[t_-,t_+]}$.
Hence $\bar\gamma|_{[t_-,t_+]}$ is not a shortest path.
Shrinking this curve to a small neighborhood of $p$
(by using homogeneity of the cone) shows that arbitrarily
small intervals of $\gamma$ near $p$ are also not shortest paths.
Therefore $\gamma$ is not a geodesic, a contradiction.
\end{proof}

Now let us compare the initial intervals $\gamma|_{[a,s_1]}$
and $\gamma'|_{[a,s'_1]}$ of $\gamma$ and $\gamma'$.
Again, we may assume that $X$ is a cone with apex $p:=\gamma(a)$.
Applying \eqref{e-strict-busemann} within a face $F$
containing both $\gamma(s_1)$ and $\gamma'(s'_1)$
and taking into account Lemma \ref{l-radial-rays},
we conclude that
$$
 \len(\gamma'|_{[a,s'_1]}) \ge \dist_F(p,\gamma'(s'_1))
 \ge \dist_F(p,\gamma(s_1)) = \len(\gamma|_{[a,s_1]}) ,
$$
and this inequality turns into equality only if
$\gamma'|_{[a,s'_1]}$ is a radial segment from $p$
and $\gamma'(s'_1)=\gamma(s_1)$. Since we are assuming
that $\gamma'$ and $\gamma$ do not have a common
initial interval, this equality case is impossible.
Thus
$$
 \len(\gamma'|_{[a,s'_1]}) > \len(\gamma|_{[a,s_1]}) .
$$
Similarly,
$$
 \len(\gamma'|_{[s'_n,b]}) \ge \len(\gamma|_{[s_n,b]}) .
$$
Summing up these two inequalities and the inequalities
from Lemma \ref{l-between-slices} yields that
$\len(\gamma')>\len(\gamma)$.
This finishes the proof of Theorem \ref{t-geodesics-are-stable}.
\end{proof}

\section{Uniqueness of geodesics in homotopy classes}
\label{uniqueness}

Let $X$ be a length space.
We say that $X$ \textit{has locally unique geodesics}
if every point of $X$ has a neighborhood such any
two points in this neighborhood are connected by
a unique shortest path.
(This is the same as Definition \ref{d:nonfocusing}
but without that assumption that the space is PL.)

A metric space $X$ is called \textit{proper} if every
closed ball in $X$ is compact.
Recall that every complete locally compact
length space is proper.

\begin{theorem}
\label{t-uniqueness-in-length-space}
Let $X$ be a proper length space
with locally unique geodesics.
Suppose that every geodesic segment in $X$
is stable.
Then there is only one geodesic segment between any two points
in each homotopy class.
\end{theorem}

This Theorem together with Theorem \ref{t-geodesics-are-stable} immediately imply
Theorem \ref{glob}. 

\begin{remark*}
The statement of the theorem is very similar
to e.g.\ the Cartan--Hadamard theorem for
locally CAT(0) spaces
or Riemannian manifolds without conjugate points.
However we face an additional difficulty here:
although every geodesic segment has a neighborhood
where it is a unique length minimizer,
the sizes of these neighborhoods are not
\textit{a priori} locally bounded away from zero.
This prevents one from proving the theorem
by a general topology argument.
For example, there could be a continuous family
$\{\gamma_\tau\}$, $\tau\in(-\ep,\ep)$, of minimizing geodesics 
starting at the same point $p\in X$ and ending
at points $q_\tau$ which behave like the values of the
function $\tau\mapsto\tau^2$ (that is, $q_{-\tau}=q_\tau$).
The geodesics in this family are locally unique
in our sense but the endpoint map is not locally injective.
\end{remark*}

\begin{proof}[Proof of Theorem \ref{t-uniqueness-in-length-space}]
For $x\in X$, denote by $\rho(x)$ the maximal radius $r$
such that every two points in the ball $B_r(x)$
are connected by a unique shortest path.
For a set $K\subset X$,
let $\rho(K)=\inf\{\rho(x):x\in K\}$.
The function $x\mapsto\rho(x)$ is obviously
continuous, hence $\rho(K)>0$ for every pre-compact set $K$.
We refer to $\rho(K)$ as the \textit{uniqueness radius} of~$K$.

A \textit{broken geodesic} is a curve composed of several
minimal geodesic segments. We will only consider broken
geodesics with sufficiently short edges  so that every
edge lies well within the uniqueness radius of their endpoints.
Such a broken geodesic is determined by the sequence of
its vertices. 

More precisely, for a positive integer $n$ let
$\Omega^n\subset X^{n+1}$ be the set of all
sequences $s=(x_0,x_1,\dots,x_n)\in X^{n+1}$
such that the maximum edge length
$\delta(s):=\max\{|x_ix_{i+1}|\}$
and the total length
$L(s):=\sum |x_ix_{i+1}|$ 
satisfy
$$
 2\delta(s) < \rho (B_{L(s)}(x_0)) .
$$
The points $x_0$ and $x_n$ are referred to as the endpoints of~$s$.
The set of $s\in\Omega^n$ with fixed endpoints $p,q\in X$
is denoted by $\Omega^n_{pq}$.
Clearly $\Omega^n$ is an open subset of $X^{n+1}$
and $\Omega^n_{pq}$ is pre-compact.
A sequence $(x_0,\dots,x_n)\in\Omega^n$ corresponds to a geodesic in $X$
if and only if $|x_{i-1}x_i|+|x_ix_{i+1}|=|x_{i-1}x_{i+1}|$
for all $i=1,\dots,n-1$.
We call such sequences $s$ \textit{geodesics}.

A point $z\in X$ is a \textit{midpoint} between
$x,y\in X$ if $|xz|=|yz|=\frac12|xz|$.
If $x$ and $y$ are sufficiently close to each other
(e.g.\ if $|xy|<\rho(x)$), then a midpoint between
$x$ and $y$ is unique and depends continuously
on $x$ and $y$.
We define the \textit{midpoint shortening} map
$T\colon\Omega^n\to\Omega^n$ as follows.
Let $s=(x_0,\dots,x_n)\in\Omega^n$.
For each $i=1,\dots,n$, let $y_i$
be the midpoint between $x_{i-1}$ and $x_i$.
By the triangle inequality we have
$$
 |y_iy_{i+1}| \le \tfrac12 (|x_{i-1}x_i|+|x_ix_{i+1}|)
$$
for all $i=1,\dots,n-1$.
In addition, $|x_0y_1|=\frac12|x_0x_1|$ and
$|y_nx_n|=\frac12|x_{n-1}x_n|$.
It follows that $|y_iy_{i+1}|\le \delta(s)$ and
the length of the broken geodesic with
vertices $x_0,y_1,\dots,y_n,x_n$ is not greater than $L(s)$.
For each $i=1,\dots,n-1$, let $x'_i$ be the midpoint
between $y_i$ and $y_{i+1}$.
We define 
$$
T(s)=(x_0,x'_1,\dots,x'_{n-1},x_n) .
$$
Applying the triangle inequality again yields
that $L(T(s))\le L(s)$.
For $1\le i\le n-2$ we have
$$
 |x'_ix'_{i+1}| \le \tfrac12 (|y_iy_{i+1}|+|y_{i+1}y_{i+2}|) 
 \le\tfrac14|x_{i-1}x_i|+\tfrac12|x_ix_{i+1}|+\tfrac14|x_{i+1}x_{i+2}|
 \le \delta(s) ,
$$
and for the first edge of $T(s)$ we have
$$
 |x_0x'_1| \le |x_0y_1| + \tfrac12|y_1y_2|
 = \tfrac12|x_0x_1| + \tfrac12|y_1y_2|
 \le \tfrac34|x_0x_1|+\tfrac14|x_1x_2|
 \le \delta(s) .
$$
Similarly, 
$$
|x'_{n-1}x_n| \le \tfrac14|x_{n-2}x_{n-1}|+\tfrac34|x_{n-1}x_n| \le \delta(s).
$$
Thus $\delta(T(s))\le \delta(s)$ and $L(T(s))\le L(s)$,
therefore $T(s)\in\Omega^n$.

This we have constructed a continuous map $T\colon\Omega^n\to\Omega^n$
which preserves endpoints and 
such that $\delta(T(s))\le\delta(s)$
and $L(T(s))\le L(s)$ for all $s\in\Omega^n$.
Note that the equality $L(T(s))=L(s)$ is attained
only if $s$ is a geodesic
(otherwise one of the intermediate
triangle inequalities is strict).

For $s=(x_0,\dots,x_n)\in\Omega^n$, let $E(s)=\sum |x_ix_{i+1}|^2$
(this quantity is similar to the energy of a path in a Riemannian manifold).
Note that $E(s)\ge \frac1n L(s)^2$ and equality is attained
if and only if the edges of $s$ have equal lengths.

\begin{lemma}
\label{l-energy}
For every $s\in\Omega^n$, we have $E(T(s))\le E(s)$.
The equality $E(T(s))=E(s)$ is attained only if $T(s)=s$,
and in this case $s$ is geodesic and all its edges
have equal lengths.
\end{lemma}

\begin{proof}
Let $l_1,\dots,l_n$ be the lengths of the edges of $s$
and $l'_1,\dots,l'_n$ the lengths of the edges of $T(s)$.
Applying the triangle inequality as in the above argument,
we obtain that
$$
\begin{aligned}
 l'_1 &\le \tfrac34 l_1 + \tfrac14 l_2, \\
 l'_i &\le \tfrac14 l_{i-1}+ \tfrac12 l_i + \tfrac14 l_{i+1}, \qquad 1\le i\le n-1, \\
 l'_n &\le \tfrac14 l_{n-1} + \tfrac34 l_n .
\end{aligned}
$$
By Jensen's inequality for the function $t\mapsto t^2$,
this implies that
$$
\begin{aligned}
 l^{\prime2}_1 &\le \tfrac34 l_1^2 + \tfrac14 l_2^2, \\
 l^{\prime2}_i &\le \tfrac14 l_{i-1}^2+ \tfrac12 l_i^2 + \tfrac14 l_{i+1}^2,
 \qquad 1\le i\le n-1, \\
 l^{\prime2}_n &\le \tfrac14 l_{n-1}^2 + \tfrac34 l_n^2 .
\end{aligned}
$$
Summing up these inequalities yields that $E(T(s))\le E(s)$.
If all these inequalities turn into equalities,
then $s$ is a geodesic and all lengths $l_i$ are equal.
Conversely, if $s$ is a geodesic with equal edge lengths,
then one easily sees from the definition of~$T$
that $T(s)=s$.
\end{proof}

\begin{lemma}
For every $s\in\Omega^n$, there exists a limit
$T^\infty(s):=\lim_{k\to\infty} T^k(s)$,
and this limit is a geodesic.
The map $T^\infty\colon\Omega^n\to\Omega^n$ is continuous.
\end{lemma}

\begin{proof}
Let $p$ and $q$ be the endpoints of $s$.
By compactness, there exist a partial limit $s_0\in X^{n+1}$
of the sequence $\{T^k(s)\}$.
Since $\{\delta(T^k(s))\}$ and $\{L(T^k(s))\}$
are non-increasing sequences, we have
$
\delta(s_0)\le\delta(s)
$
and 
$
L(s_0)\le L(s) ,
$
hence $s_0\in\Omega^n$.
Similarly, Lemma \ref{l-energy} implies that
$E(s_0)=\inf_k E(T^k(s))$, hence $E(T(s_0))=E(s_0)$
and therefore $s_0$ is a geodesic with equal edge lengths
and $T(s_0)=s_0$.

Since all geodesic in $X$ are stable,
there is a neighborhood $U_0$ of $s_0$
such that $E(s')>E(s_0)$ for all $s'\in U_0\setminus\{s_0\}$.
Choose a neighborhood $U\Subset U_0$ of $s_0$.
Since $T(s_0)=s_0$, there is a neighborhood $U_1$ of $s_0$
such that $T(U_1)\subset U$.
Let
$$
E_0 = \inf \{ E(s') : s'\in U\setminus U_1\}
$$
and note that $E_0>E(s_0)$.
Let $U_2= \{ s'\in U:E(s')<E_0\}$,
then $U_2\subset U_1$ by the definition of $E_0$.
Observe that $T(U_2)\subset U_2$
since for every $s'\in U_2$ we have
$T(s')\in U$ by the choice of $U_1$ and
$E(T(s'))\le E(s')<E_0$ by Lemma \ref{l-energy}.

Since $T(U_2)\subset U_2$ and $U_2$ contains
some members of the sequence $\{T^k(s)\}$,
a tail of this sequence is contained in~$U_2\subset U$.
Since $U$ can be chosen arbitrarily small,
it follows that $T^k(s)\to s_0$ as $k\to\infty$.

Thus we have shown that the map $T^\infty=\lim_{k\to\infty} T^k$
is well-defined, let us show that it is continuous.
For $s\in\Omega^n$ and $s_0=T^\infty(s)$,
let $U$ and $U_2$ be as above.
There exists $k_0$ such that $T^{k_0}(s)\in U_2$.
Since $T^{k_0}$ is continuous, there is a neighborhood
$U'$ of $s$ such that $T^{k_0}(U')\subset U_2$.
Since $T(U_2)\subset U_2$, it follows that
$T^k(U')\subset U_2\subset U$ for all $k>k_0$
and therefore $T^\infty(U')$ is contained in
the closure of~$U$.
Since $U$ is an arbitrarily small neighborhood of $s_0$,
it follows that $T^\infty$ is continuous at $s$.
\end{proof}

Now we are in a position to prove the theorem.
Suppose that two geodesics
$\gamma_0,\gamma_1\colon [0,1]\to X$
connecting points $p$ and $q$
are homotopic via a homotopy
$H=\{\gamma_\tau\}_{\tau\in[0,1]}$
such that $\gamma_\tau(0)=p$
and $\gamma_\tau(1)=q$ for all $\tau$.
Choose a large $N$ and replace each path $\gamma_\tau$
by a broken geodesic $\bar\gamma_\tau$
with vertices at points $\gamma_\tau(i/N)$,
$i=0,1,\dots,N$. If $N$ is large enough,
the edges of these broken geodesics lie within
the uniqueness radius of the image of~$H$
and therefore vary continuously in~$\tau$.
Let $L_0$ be the maximum of lengths of $\bar\gamma_\tau$
and $\rho=\rho(B_{L_0}(p))$.
Subdividing each edge of $\bar\gamma_\tau$
into $M$ equal pieces where $M>2\rho^{-1}L_0$
yields a sequence $s_\tau\in\Omega^{M}_{pq}$.
Now we can apply our shortening procedure:
let $\bar s_\tau=T^\infty(s_\tau)$.
The family $\{\bar s_\tau\}$ is a continuous
family of geodesics connecting $p$ and~$q$.
Since all geodesics are stable,
the (continuous) function $\tau\mapsto L(\bar s_\tau)$
possesses an impossible property:
it has a strict local minimum at every point
$\tau\in[0,1]$. This contradiction proves
Theorem~\ref{t-uniqueness-in-length-space}.
\end{proof}

\section{Saddle Surfaces}
\label{saddle}

In this section we prove an analog of the main result from 
\cite{BI} for PL surfaces.

Let $V$ be a finite-dimensional normed space with
a $C^1$ smooth strictly convex norm. By a \textit{PL surface}
in $V$ we mean a PL map from a triangulated two-dimensional
manifold $M$ to~$V$ such that the image of every triangle
in $M$ is not degenerate. Such a map induces the structure
of a Finsler PL space on $M$.

Recall that a two-dimensional surface in a vector space
is \textit{saddle} (cf.~\cite{Shefel}) if one cannot cut off
a cap from the surface by a hyperplane.
Note that a PL surface $r\colon M\to V$ is saddle if
for every $x\in M$ and every neighborhood $U$ of $x$,
the convex hull of $r(U\setminus\{x\})$ contains $r(x)$.

\begin{theorem}
\label{cone}
Let $V$ be a 
finite-dimensional normed space with
a $C^1$ smooth strictly convex norm
and $r\colon M\to V$ a saddle PL surface.
Then the PL Finsler metric on~$M$ induced by $r$
is non-focusing.

Therefore if $M$ is simply connected then every two points
are connected by a unique geodesic.
\end{theorem}

\begin{proof}
Since the statement is local, we may assume that our surface is a cone
with the apex at the origin.
That is, $M=\R^2$ and $r\colon\R^2\to V$ is a positive homogeneous
piecewise linear map.
We divide $\R^2$ by several radial rays
(that is, rays starting at the origin)
into convex sectors such that $r$ is linear
on each sector. These rays (and their images under $r$)
are referred to as \textit{edges}
and the sectors as \textit{faces}
of our surface.

We regard $\R^2$ with the PL Finsler metric induced by $r$. 
Every shortest path in this metric
is either a segment of a radial ray,
or a union of two segments with a common endpoint
at the origin, or a broken line that intersects every
radial ray at most once.

Suppose that there are two
points $p,q\in\R^2$ and two distinct shortest paths
connecting $p$ and $q$.
Denote these paths by $\gamma_0$ and $\gamma_1$.
We may assume that $\gamma_0$ and $\gamma_1$
have no common points except $p$ and $q$.
Note that $p$ and $q$ are distinct from the origin,
and $\gamma_1$ and $\gamma_2$ cannot be
mapped by $r$ to straight line segments in $V$
(since all segments are unique shortest paths in $V$).

Consider the following cases:

\textit{Case 1}:
$\gamma_0$ contains the origin
(and therefore $\gamma_0=[p,0,q]$).
Then $\gamma_1$ is a broken line
of the form $[p,x_1,x_2,\dots,x_k,q]$
where $x_1,\dots,x_k$ are some points on edges.
We include $\gamma_0$ and $\gamma_1$ in
a family of broken lines $\{\gamma_t\}_{t\in[0,1]}$
where
$$
 \gamma_t=[p,tx_1,tx_2,\dots,tx_k,q] .
$$
Then
$$
 \len(\gamma_t) = \|r(p)-tr(x_1)\| + \|r(q)-tr(x_k)\| 
 + t\cdot\sum_{i=1}^{k-1} \|r(x_i)-r(x_{i+1})\|.
$$
The first two terms are strictly convex functions in $t$
and the third term is linear, hence
the function $t\mapsto \len(\gamma_t)$
is strictly convex.
Since $\len(\gamma_0)=\len(\gamma_1)$, it follows that
$\len(\gamma_{1/2})<\len(\gamma_0)$,
hence $\gamma_0$ is not a shortest path.

\textit{Case 2}: the region enclosed by $\gamma_0$
and $\gamma_1$ does not contain the origin.
Then $\gamma_0$ and $\gamma_1$ intersect the
same set of edges and intersect the edges in the same order.
Let $\gamma_0=[p=x_0,x_1,\dots,x_k=q]$
and $\gamma_1=[p=y_0,y_1,\dots,y_k=q]$
where $x_i$ and $y_i$ ($i=1,\dots,k-1$)
are points on the edges. As noted above,
$x_i$ and $y_i$ are on the same edge
for every $i=1,\dots,k-1$,
hence each pair of segments $[x_i,x_{i+1}]$
and $[y_i,y_{i+1}]$, where $0\le i\le k-1$
is contained in one face of our surface.
Similarly to the first case,
include $\gamma_0$ and $\gamma_1$
in a family $\{\gamma_t\}_{t\in[0,1]}$
of broken lines defined by
$$
 \gamma_t = [s_0(t),s_1(t),\dots,s_k(t)]
$$
where $s_i(t)= (1-t)x_i+ty_i$.
Then
$$
 \len(\gamma_t) = \sum_{i=0}^{k-1} \|r(s_i(t))-r(s_{i+1}(t))\| .
$$
Each term of the sum is a convex function of $t$
and at least one of them (e.g.\ the one for $i=0$)
is strictly convex. Hence the sum is a strictly convex
function of $t$, and we obtain a contradiction
as in Case~1.

\textit{Case 3}: the region enclosed by $\gamma_0$ and $\gamma_1$
contains the origin. In this case every radial ray in $\R^2$
intersects the union of $\gamma_1$ and $\gamma_2$.

Since our surface is saddle, the origin 0 of $V$ is
contained in the convex hull of the set $r(\R^2)\setminus\{0\}$.
Therefore 0 is a positive linear combination of
a finite collection of nonzero vectors from the image of~$r$.
Since the surface is a cone, these vectors
can be rescaled so that the coefficients
of the linear combination are equal to~1.
Thus there exist vectors $v_1,\dots,v_k\in r(\R^2\setminus\{0\})$
such that $\sum v_i=0$.

Define a function $f:V\to\R$ by
$$
 f(x)=\|x-r(p)\|+\|x-r(q)\| .
$$
This function is convex and differentiable outside the set
$\{r(p),r(q)\}$, in particular, it has
a derivative $d_0f$ at the origin.
Since $\sum v_i=0$, we have  $\sum d_0f(v_i)=0$, 
therefore there exists $j\in\{1,\dots,k\}$ such that
$d_0f(v_j)\ge 0$. Since $f$ is convex, it follows that
$f(tv_j)\ge f(0)$ for all $t\ge 0$.

The ray $\{tv_j:t\ge 0\}$ is the $r$-image of
some radial ray $\ell\subset\R^2$.
Recall that every radial ray in $\R^2$
intersects at least one of the shortest paths $\gamma_0$ and $\gamma_1$.
Without loss of generality assume that a point
$z\in\ell$ lies on $\gamma_0$. Then
$$
 \len(\gamma_0) \ge \|r(p)-r(z)\|+\|r(z)-r(q)\| = f(r(z)) \ge f(0)
$$
since $r(z)\in\{tv_j:t\ge 0\}$.
Note that $f(0)$ is the length of the broken line $[r(p),0,r(q)]$
which is the $r$-image of the broken line $[p,0,q]$.
Thus $\len(\gamma_0)\ge \len[p,0,q]$.
Since $\gamma_0$ is a shortest path,
it follows that $\len(\gamma_0)=\len[p,0,q]$.
Now we can replace $\gamma_0$ by $[p,0,q]$
and thus reduce Case~3 to Case~1.

Now the second statement of the Theorem follows immediately
from Theorem~\ref{glob}.
\end{proof}

\bibliographystyle{plain}

\end{document}